\pdfoutput=1
\documentclass[11pt]{article}
\usepackage{amssymb}
\usepackage[all,cmtip]{xy}
\usepackage[width=17cm,top = 3.5cm,bottom=3.5cm]{geometry}
\usepackage{tikz}
\usepackage{tikz-cd}
\usepackage[all]{xy}
\usetikzlibrary{matrix,arrows,decorations.pathmorphing}
\usepackage{amsthm,amsfonts,amssymb,amsmath,amscd} 
\usepackage{aliascnt} 
\usepackage{mathrsfs} 
\usepackage{xargs,ifthen} 

\usepackage{comment}
\usepackage{color}

\usepackage{hyperref}
\definecolor{tocolor}{rgb}{.1,.1,.1}
\definecolor{urlcolor}{rgb}{.2,.2,.6}
\definecolor{linkcolor}{rgb}{.1,.1,.5}
\definecolor{citecolor}{rgb}{.4,.2,.1}
\hypersetup{backref=true, colorlinks=true, urlcolor=urlcolor, linkcolor=linkcolor, citecolor=citecolor}

\newcommandx{\thdef}[2]{
	\newaliascnt{#1}{theorem}  
	\newtheorem{#1}[#1]{#2}
	\aliascntresetthe{#1}  
	\newtheorem*{#1*}{#2}
	\expandafter\newcommand\expandafter{\csname #1autorefname\endcsname}{#2}
}

\makeatletter
\newtheorem*{rep@theorem}{\rep@title}
\newcommand{\newreptheorem}[2]{%
\newenvironment{rep#1}[1]{%
 \def\rep@title{#2 \ref{##1}}%
 \begin{rep@theorem}}%
 {\end{rep@theorem}}}
\makeatother

\newtheorem{theorem}{Theorem}[section]
\newreptheorem{theorem}{Theorem}

\thdef{lemma}{Lemma}
\thdef{corollary}{Corollary}
\thdef{conjecture}{Conjecture}
\newreptheorem{conjecture}{Conjecture}
\thdef{proposition}{Proposition}
\thdef{question}{Question}

\theoremstyle{definition}
\thdef{definition}{Definition}
\thdef{notation}{Notation}

\theoremstyle{remark}
\thdef{remark}{Remark}

\theoremstyle{remark}
\thdef{ex}{Example}

\newenvironment{example}
{\begin{ex}}%
{\hfill $\blacksquare$\end{ex}}

\newcommand{\spc}[1]{\mathsf{#1}} 
\newcommand{\shf}[1]{\mathcal{#1}} 

\newcommand{\rbrac}[1]{\left(#1\right)} 

\newcommandx{\fn}[2][2=]{#1\ifthenelse{\equal{#2}{}}{}{\!\rbrac{{#2}}}} 
\newcommandx{\id}[2][2=]{\fn{{\rm id}_{#1}}[#2]} 

\newcommand{\ext}[2][\bullet]{\spc{\Lambda}^{#1}{#2}} 
\newcommandx{\End}[2][1=]{\fn{\spc{End}_{#1}}[#2]} 
\newcommandx{\Hom}[2][1=]{\fn{\spc{Hom}_{#1}}[#2]} 
\newcommandx{\Aut}[2][1=]{\fn{\spc{Aut}_{#1}}[#2]} 
\newcommandx{\image}[1]{\fn{\spc{img}}[#1]} 
\renewcommandx{\ker}[1]{\fn{\spc{ker}}[#1]} 
\newcommandx{\rank}[1]{\fn{\mathrm{rank}}[#1]} 

\newcommandx{\ann}[1]{\fn{\spc{ann}}[#1]} 

\newcommandx{\hlgy}[3][1=\bullet,3=]{\spc{H}_{#1}^{#3}\!\rbrac{{#2}}} 
\newcommandx{\cohlgy}[3][1=\bullet,3=]{\spc{H}^{#1}_{#3}\!\rbrac{{#2}}} 
\newcommandx{\chow}[3][1=\bullet,3=]{\spc{A}^{#1}_{#3}\!\rbrac{{#2}}} 

\newcommandx{\Ext}[3][1=\bullet,3=]{\fn{\spc{Ext}^{#1}_{#3}}[{#2}]} 
\newcommandx{\Tor}[3][1=\bullet,3=]{\fn{\spc{Tor}^{#1}_{#3}}[{#2}]} 

\newcommandx{\Pic}[1]{\fn{\spc{Pic}}[{#1}]} 
\newcommandx{\chernalg}[2][1=\bullet]{\fn{\spc{Chern}^{#1}}[{#2}]} 
\newcommandx{\chern}[2][1=]{\fn{c_{#1}}[#2]} 
\newcommandx{\ch}[2][1=]{\fn{\mathrm{ch}_{#1}}[{#2}]} 

\newcommandx{\sKer}[2][1=]{ \fn{ \shf{K}er_{#1}}[{#2}] } 
\newcommandx{\sHom}[2][1=]{ \fn{ \shf{H}om_{#1}}[{#2}] } 
\newcommandx{\sEnd}[2][1=]{ \fn{ \shf{E}nd_{#1}}[{#2}] } 
\newcommandx{\sExt}[3][1=\bullet,3=]{\fn{\shf{E}xt^{#1}_{#3}}[{#2}]} 
\newcommandx{\sTor}[3][1=\bullet,3=]{\fn{\shf{T}or^{#1}_{#3}}[{#2}]} 

\newcommandx{\forms}[2][1=\bullet]{\Omega^{#1}_{#2}} 
\newcommandx{\can}[1][1=]{\omega_{#1}} 
\newcommandx{\acan}[1][1=]{\omega_{#1}^{-1}} 
\newcommandx{\tshf}[1]{\shf{T}_{#1}} 
\newcommandx{\mvect}[2][1=\bullet]{ \ext[#1]{\tshf{#2}} }
\newcommandx{\der}[2][1=\bullet]{\mathscr{X}^{#1}_{#2}} 
\newcommandx{\sJet}[3][1=,2=]{\shf{J}^{#1}_{#2}#3} 

\newcommandx{\tb}[2][1=]{\spc{T}_{\!#1}{#2}} 
\newcommandx{\ctb}[2][1=]{\spc{T}_{\!#1}^*{#2}} 
\newcommandx{\lie}[2][2=]{\fn{\mathscr{L}_{#1}}[#2]} 
\newcommandx{\hook}[2][2=]{\fn{i_{#1}}[#2]} 

\newcommand{\Rep}[1]{\fn{\cat{R}ep}[#1]}

\makeatletter
\newcommand{\thickbar}{\mathpalette\@thickbar}
\newcommand{\@thickbar}[2]{{#1\mkern1.5mu\vbox{
  \sbox\z@{$#1\mkern-1mu#2\mkern-1mu$}%
  \sbox\tw@{$#1\overline{#2}$}%
  \dimen@=\dimexpr\ht\tw@-\ht\z@-.6\p@\relax
  \hrule\@height.4\p@ 
  \vskip1\p@
  \hrule\@height.4\p@ 
  \vskip\dimen@
  \box\z@}\mkern1.5mu}
}
\makeatother

\def\Rep{\text{Rep}}

\def\ker{\text{ker}}

\def\End{\text{End}}
\def\log{\text{log}}

\newcommand{\bb}[1]{\mathbb{#1}}

\numberwithin{equation}{section}
\newtheoremstyle{parag}
  {\topsep}   
  {\topsep}   
  {}  
  {}       
  {\bfseries} 
  {.}         
  { } 
  {}          
\theoremstyle{parag}

\usepackage{MnSymbol}

\makeatletter
\def\@cite#1#2{{\normalfont[{#1\if@tempswa , #2\fi}]}}
\makeatother

\renewcommand{\Pic}{\mathrm{Pic}}

\def\Aut{\text{Aut}}

\begin{document}

\title{\vspace{-4em} \huge Castling equivalence for logarithmic flat connections}

\date{}

\author{
Francis Bischoff\thanks{Exeter College and Mathematical Institute, University of Oxford; {\tt francis.bischoff@maths.ox.ac.uk }}
}
\maketitle
\abstract{Let $X$ be a complex manifold containing a hypersurface $D$ and let $D^s$ denote the singular locus. We study the problem of extending a flat connection with logarithmic poles along $D$ from the complement $X \setminus D^s$ to all of $X$. In the setting where $D$ is a weighted homogeneous plane curve we give a new proof of Mebkhout's theorem that extensions always exist. Our proof makes use of a Jordan decomposition for logarithmic connections as well as a version of Grothendieck's decomposition theorem for vector bundles over the `football' orbifold which is due to  Martens and Thaddeus. In higher dimensions we point out a close relationship between the extension problem and castling equivalence of prehomogeneous vector spaces. In particular, we show that the twisted fundamental groupoids of castling equivalent linear free divisors are `birationally' Morita equivalent and we use this to generate examples of non-extendable flat connections. }

\tableofcontents

\section{Introduction}
Let $X$ be a complex manifold containing a hypersurface $D$. A version of the Riemann-Hilbert problem asks whether a flat connection on the complement $U = X \setminus D$ admits an extension to a flat connection on $X$ with logarithmic singularities along $D$. In the case where $X = \bb{CP}^1$ is the complex projective line and $D$ is a collection of points, this question was posed by Hilbert and has been extensively studied (eg. \cite{MR0174815, MR0090529, MR0982153, MR0086958,  MR1032917}). When the hypersurface $D$ is smooth, or has simple normal crossing singularities, there is a construction due to Deligne (attributed to Manin) \cite{deligne2006equations, MR882000} which provides a canonical extension given the choice of a (set-theoretic) branch of the logarithm.\footnote{Because this construction uses matrix logarithms, it only works when the structure group is $GL(m, \bb{C})$. For other structure groups there are counterexamples to the Riemann-Hilbert problem even when the divisor is smooth.} In this paper we study the following different but related \emph{extension problem} in the setting where $D$ is a singular Saito free divisor \cite{saito1980theory}. Let $D^{s}$ denote the singular locus of $D$, let $X^\times = X \setminus D^{s}$ and let $D^{\times} = D \setminus D^{s}$. Hence $D^{\times}$ is a smooth hypersurface in the complex manifold $X^\times$. The spaces involved are arranged according to the following diagram of inclusions. 
\[
\begin{tikzcd} [column sep=-3pt,row sep=10pt]
D^s & = & D^s &  \\
D \arrow[u, phantom, sloped, "\supset"]& \subset & X \arrow[u, phantom, sloped, "\supset"]& \supset & U \\
D^\times \arrow[u, phantom, sloped, "\subset"]& \subset & X^\times \arrow[u, phantom, sloped, "\subset"] & \supset & U \arrow[u, phantom, sloped, "\equal"]
\end{tikzcd}
\]

\begin{question*}[Extension problem]
Let $(E, \nabla)$ be a logarithmic flat connection on $(X^\times, D^{\times})$. Is there an extension $(\tilde{E}, \tilde{\nabla})$ to a logarithmic flat connection on $(X,D)$?
\end{question*}

Note that by Hartogs' theorem, the extension is unique and its existence depends only on whether the vector bundle $E$ can be extended. By applying the Deligne-Manin construction, any flat connection on $U$ may be extended to a logarithmic flat connection $(E, \nabla)$ on $(X^\times, D^{\times})$, since $D^{\times}$ is smooth. Hence, solving the extension problem in this case immediately leads to a solution of the Riemann-Hilbert problem. However, the two problems are also quite distinct: there are logarithmic flat connections $(E, \nabla)$ which fail to extend even though their restriction to $U$ \emph{does} extend. 

An illustrative example of this phenomenon is provided by $(\mathbb{C}^n, D)$, where $D$ is the union of the coordinate hyperplanes. In this setting, logarithmic flat connections with trivial monodromy are equivalent to $(\mathbb{C}^*)^n$-equivariant vector bundles, and these have been completely classified as a special case of the study of equivariant vector bundles over toric varieties \cite{MR1011361, MR1171283, MR1024452, MR376680, MR961215, MR1663340, MR1729338}. In this classification, an equivariant vector bundle over $\mathbb{C}^n \setminus D^{s}$ is equivalent to the data of a vector space $V$ which is equipped with an $n$-tuple of decreasing $\mathbb{Z}$-filtrations $F_{1}, F_{2}, ..., F_{n}$. This equivariant bundle extends over $\mathbb{C}^n$ if and only if the tuple of filtrations can be simultaneously split. Note that restricting the bundle to $\mathbb{C}^n \setminus D$ corresponds to throwing away the filtrations. This restricted bundle can then be extended all the way to $\mathbb{C}^n$ by now choosing an $n$-tuple of simultaneously split filtrations (for example, choosing $n$-copies of a single filtration). In other words, this provides examples of non-extendable connections on $(\mathbb{C}^n \setminus D^{s}, D^\times)$ whose restriction to $U = \mathbb{C}^n \setminus D$ does nevertheless admit an extension. 

\subsection*{Mebkhout's extension}
The case $n = 2$ is noteworthy. Indeed, since it is always possible to simultaneously split a pair of filtrations, equivariant bundles over $\mathbb{C}^2 \setminus \{ (0,0) \}$ may always be extended. In fact, this is a special case of a theorem of Mebkhout \cite{MR2077649} (see also a similar result by Kita \cite{MR0541894}), who proved that the extension problem can always be solved when $D$ is a reduced curve in $\mathbb{C}^2$. The first purpose of this paper is to give a new proof of this result in the case where $D$ is a weighted homogeneous plane curve. This is the content of Theorem \ref{extensiontheorem}. Whereas Mebkhout's original proof relied on the theory of coherent sheaves, our proof makes use of a groupoid theoretic approach to logarithmic connections which was developed in \cite{bischoff2022normal}. Briefly, our proof proceeds in two steps: 
\begin{enumerate}
\item Using a Jordan decomposition theorem, we show in Proposition \ref{inducedCstaraction} that any logarithmic flat connection $(E, \nabla)$ admits the structure of a $\mathbb{C}^*$-equivariant bundle. Note that this is the same $\mathbb{C}^*$-action with respect to which $D$ is weighted homogeneous. As a result, $E$ descends to a vector bundle on a `football' orbifold $\mathbb{P}^{1}_{p,q}$. 
\item Applying Martens and Thaddeus' version of the Grothendieck decomposition theorem for vector bundles over the football orbifolds \cite{martens2012variations} we decompose $E$ into a sum of line bundles, which can easily be seen to extend to $\mathbb{C}^2$. 
\end{enumerate}
From this perspective, the existence of non-extendable connections may be attributed to the failure of the Grothendieck decomposition theorem in higher dimensions. Indeed, Proposition \ref{inducedCstaraction} continues to hold in higher dimensions, implying that there is a close relationship between the extension problem for weighted homogeneous hypersurfaces and the study of vector bundles on weighted projective space. 

\subsection*{Castling equivalence}
In Section \ref{CastlingEq} we turn our attention to the second purpose of this paper, which is to point out a relationship between the extension problem and castling equivalence for prehomogeneous vector spaces. Castling equivalence is an operation on linear representations which first arose in the work of Sato and Kimura on the classification of prehomogeneous vector spaces \cite{MR430336}. It was related to free divisors in the work of \cite{MR2228227, MR2795728, MR2521436, MR3237442}, who introduced \emph{linear} free divisors (a special case of prehomogeneous vector spaces) and showed that they are preserved by castling equivalence. A noteworthy application of their result is that castling equivalence generates infinitely many new examples of linear free divisors. 

Now suppose that $(X_{1},D_{1})$ and $(X_{2}, D_{2})$ are castling equivalent linear free divisors and let $(X^\times_{i}, D_{i}^{\times})$ denote the respective complements of their singular loci. Each pair $(X^\times_{i}, D^{\times}_{i})$ gives rise to a holomorphic Lie groupoid $\Pi(X^\times_{i}, D^{\times}_{i})$ which has the property that its category of representations is equivalent to the category of flat connections on $X^\times_{i}$ with logarithmic singularities along $D^\times_{i}$. Our main result is Theorem \ref{castleMorita}, which states that \emph{twisted fundamental groupoids} derived from castling equivalent linear free divisors are Morita equivalent. In particular, they have equivalent categories of representations. The significance of this theorem is that it implies that the categories of logarithmic flat connections for $(X_{1},D_{1})$ and $(X_{2}, D_{2})$ can both be embedded into a common category. Furthermore, by transporting a flat logarithmic connection for $(X_{2}, D_{2})$ along the Morita equivalence, we obtain a flat logarithmic connection for $(X^{\times}_{1}, D^{\times}_{1})$ which \emph{may fail to extend}. This is demonstrated in Proposition \ref{castlingembedding} which provides a simple method to construct non-extendable connections from the data of representations of special linear groups. 

\vspace{.1in}

\noindent \textbf{Acknowledgements.} 
I would like to thank L. Narv\'{a}ez Macarro for informing me of Mebkhout's result. I would also like to thank M. Gualtieri for suggesting several improvements to the paper. 

\section{Homogeneous Lie groupoids} \label{homogeneousgrps}
In this section we briefly recall homogeneous groupoids and some of their representation theory, as developed in \cite{bischoff2022normal}. Let $X$ be a complex manifold equipped with a $(\mathbb{C}^*)^k$-action and let $\mathcal{G} \rightrightarrows X$ be a holomorphic Lie groupoid. The action determines a source-simply connected action groupoid $\mathbb{C}^{k} \ltimes X$ and we assume that there is a groupoid morphism 
\[
i : \mathbb{C}^k \ltimes X \to \mathcal{G}. 
\]
There is a natural projection morphism $p: \mathbb{C}^k \ltimes X \to \mathbb{C}^k$ and we assume that we have an extension of this morphism to $\mathcal{G}$: 
\[
\pi : \mathcal{G} \to \mathbb{C}^k. 
\]
Note that the subgroup $2 \pi i \mathbb{Z}^k \subset \mathbb{C}^k$ determines an isotropic subgroupoid $\mathbb{Z}^k \times X \subset \mathcal{G}$. 

\begin{definition}
Let $(\mathcal{G}, i, \pi)$ be a triple as above. This data 
\begin{itemize}
\item is \emph{central} if for all $n \in \mathbb{Z}^k$ and $g \in \mathcal{G}$, we have 
\[
(n, t(g)) \ast g = g \ast (n, s(g)). 
\]
\item has a \emph{unique s-equivalence class} if every point of $X$ is equivalent under the equivalence relation generated by the $\mathcal{G}$-orbit closures. 
\end{itemize}
The data $(\mathcal{G}, i, \pi)$ defines a \emph{homogeneous groupoid} if it is both central and has a unique $s$-equivalence class. It is furthermore called \emph{positive} in the case that $X$ is a vector space and one of the $\mathbb{C}^{*}$-actions has strictly positive weights. 
\end{definition}

In \cite{bischoff2022normal} several examples of homogeneous groupoids are described. They arise from representations of algebraic groups and as the twisted fundamental groupoids of weighted homogeneous Saito free divisors. Let us recall a proposition which will be used in a future section. 
\begin{proposition}\cite[Proposition 3.4]{bischoff2022normal} \label{equivgroupoidstr}
Let $(\mathcal{G}, i, \pi)$ be a homogeneous groupoid. Then the conjugation action of $\mathbb{C}^{k}$ descends to an action of $(\mathbb{C}^{*})^{k}$ on $\mathcal{G}$ by Lie groupoid automorphisms. Furthermore, the morphism $\pi$ is invariant under this action. 
\end{proposition}

\subsection*{Monodromy and the Jordan decomposition}
Let $P \to X$ be a right principal $H$-bundle, where $H$ is a connected complex reductive group. Recall that the $H$-equivariant isomorphisms between the fibres of $P$ assemble into the Atiyah groupoid $\mathcal{G}(P)$, which is a Lie groupoid over $X$. A representation of the groupoid $\mathcal{G}$ consists of a principal bundle $P$ and a morphism $\phi: \mathcal{G} \to \mathcal{G}(P)$ covering the identity. In the case of a homogeneous groupoid we can restrict to the subgroupoid $\mathbb{Z}^{k} \times X$ to get the monodromy, which we view as a map 
\[
M : \mathbb{Z}^{k} \to Aut_{H}(P), 
\]
where $Aut_{H}(P)$ is the group of gauge transformations of $P$. In fact, the image of the monodromy lies in the subgroup $Aut(\phi)$ of automorphisms of $\phi$. Let $S$ and $U$ be the semisimple and unipotent components of $M$, respectively, defined according to the multiplicative Jordan-Chevalley decomposition. In \cite[Lemma 4.3]{bischoff2022normal} we show that $S$ and $U$ are well-defined holomorphic automorphisms of $\phi$, and that the conjugacy class of $S_{n}$, for $n \in \mathbb{Z}^k$, is constant over $X$. We furthermore prove the following functorial Jordan decomposition theorem. Let $Rep(\mathcal{G}, H)$ be the category of $H$-representations of $\mathcal{G}$, and let $\mathcal{JC}_{H}$ denote the category whose objects are triples $(P, \phi_{s}, U)$, where $(P, \phi_{s})$ is a representation with semisimple monodromy, and $U: \mathbb{Z}^{k} \to Aut(\phi_{s})$ is a morphism valued in unipotent automorphisms. 
\begin{theorem}\cite[Theorem 4.6]{bischoff2022normal} \label{JCdecomp}
There is an isomorphism of categories 
\[
Rep(\mathcal{G}, H) \cong \mathcal{JC}_{H}. 
\]
\end{theorem}
Let us briefly describe the isomorphism. Given a representation $(P, \phi)$ we assign to it the triple $(P, \phi_{s}, U)$, where $U$ is the unipotent part of the monodromy of $\phi$, and $\phi_{s} = \sigma_{U} \phi$, where $\sigma_{U}$ is a groupoid $1$-cocycle which `untwists' the unipotent monodromy. 

\subsection*{Equivariant structures}
We now describe an application of Theorem \ref{JCdecomp} which allows us to construct torus equivariant bundles from groupoid representations. We start by assuming that the homogeneous groupoid has the form $\mathbb{C} \ltimes X \to \mathcal{G} \to \mathbb{C}$. Note that any homogeneous groupoid gives rise to one of this form by forgetting some of the $\mathbb{C}^*$-actions. For simplicity, we will assume that the structure group is $H = GL(m, \mathbb{C})$. 

Recall from Proposition \ref{equivgroupoidstr} that there is an induced $\mathbb{C}^{*}$-action on the groupoid $\mathcal{G}$. The following result allows us to lift this action to any representation of $\mathcal{G}$. 

\begin{proposition} \label{inducedCstaraction}
Let $(P, \phi)$ be a $GL(m, \mathbb{C})$-representation of $\mathcal{G}$ and choose $S \in GL(m, \mathbb{C})$ in the conjugacy class of the semisimple component of the monodromy of $\phi$. Then 
\begin{itemize}
\item $(P, \phi)$ admits a reduction of structure group to the centralizer $C_{GL(m, \mathbb{C})}(S)$. Denote the resulting representation $(K, \phi)$. 
\item The $\mathbb{C}^*$-action on $X$ (and $\mathcal{G}$) lifts to an equivariant action on $K$ which acts by automorphisms of $\phi$. 
\end{itemize}
In particular, $P$ admits the structure of a $\mathbb{C}^*$-equivariant bundle over $X$. 
\end{proposition}
\begin{proof}
First, by Theorem \ref{JCdecomp}, the representation $\phi$ admits a factorisation $(P, \phi_{s}, U)$, where $\phi_{s}$ has semisimple monodromy. More precisely, $\psi = i^*(\phi_{s}) : \mathbb{C} \ltimes X \to \mathcal{G}(P)$ is a representation and the monodromy $M(x) = \psi(2 \pi i,x)$ is a semisimple automorphism of $\phi$. By \cite[Lemma 4.3]{bischoff2022normal}, the conjugacy class of $M$ is constant along $X$. Let $S \in GL(m, \mathbb{C})$ be a representative of this element and define 
\[
K = \{ p \in P \ | \ M p = p \ast S \}. 
\]
We claim that $K$ is a holomorphic reduction of structure group to $C_{GL(m,\bb{C})}(S)$, the centralizer of $S$. Indeed, it is possible to choose a local trivialization $P|_{U} \cong U \times GL(m, \bb{C})$ such that $M|_{U} = S$. Then 
\[
K|_{U} \cong \{ (u, h) \in U \times GL(m, \bb{C}) \ | \ S h = h S \} = U \times C_{GL(m,\bb{C})}(S). 
\]
Note that since $M$ is an automorphism of $\phi$, it follows that $\phi$ preserves $K$. The representation $\psi$ defines an equivariant $\mathbb{C}$-action on $K$ (or $P$) which lifts the given action on $X$. We will now modify it so that it descends to a $\mathbb{C}^*$-action. Let $A \in Lie(Z(C_{GL(m,\bb{C})}(S)))$, the Lie algebra of the centre, such that $\exp(2 \pi i A) = S$. Now define 
\[
\tilde{\psi} : \mathbb{C} \ltimes X \to \mathcal{G}(K), \qquad (\lambda, x) \mapsto R_{\exp(-\lambda A)} \psi(\lambda, x), 
\]
where $R_{\exp(-\lambda A)} : K \to K$ denotes the right action of $\exp(-\lambda A) \in C_{GL(m,\bb{C})}(S)$ on $K$. Note that the map $\tilde{\psi}(\lambda, x) : K_{x} \to K_{\lambda \ast x}$ is $C_{GL(m,\bb{C})}(S)$-equivariant precisely because $\exp(-\lambda A)$ lies in the centre of this group. Hence this defines a holomorphic action. For $p \in K$ 
\[
\tilde{\psi}(2\pi i, x)(p) = \psi(2 \pi i,x)(p) \ast e^{-2 \pi i A} = M(x)(p) \ast S^{-1} = p.
\]
Therefore, the action descends to an action of $\mathbb{C}^*$. In order to show that $\mathbb{C}^*$ acts by automorphisms of $\phi$, we need to check that for all $\mu \in \mathbb{C}$, $g \in \mathcal{G}$ and $p \in K_{s(g)}$, the following equation is satisfied 
\[
\mu \ast \phi(g)(p) = \phi(\mu \ast g) (\mu \ast p). 
\]
Unpacking the definitions, this equation is given by 
\[
\phi_{s}(\mu, t(g)) \phi(g)(p) \ast e^{-\mu A} = \phi( (\mu, t(g)) g (\mu, s(g))^{-1}) \phi_{s}(\mu, s(g))(p) \ast e^{-\mu A}. 
\]
Cancelling off the factor $e^{-\mu A}$, and using the definition of $\phi_{s}$ from \cite[Section 4.2]{bischoff2022normal}, the equation simplifies to 
\[
c(\mu)_{\mu \ast t(g)} \phi( (\mu, t(g)) g ) = \phi( (\mu, t(g)) g (\mu, s(g))^{-1}) c(\mu)_{\mu \ast s(g)} \phi( \mu, s(g)),
\]
which follows from the fact that $c(\mu) \in Aut(\phi)$. 
\end{proof}

Taking a closer look at the proof of Proposition \ref{inducedCstaraction}, we see that there is only one place where it was necessary to assume that the structure group $H = GL(m,\bb{C})$. Namely, we must be able to find a logarithm of $S$ in the centre of its centralizer $\log(S) \in Lie(Z(C_{GL(m,\bb{C})}(S)))$. Such logarithms exist because the eigenspace decomposition of $S$ implies that $C_{GL(m,\bb{C})}(S)$ is a product of general linear groups and so its centre is a connected torus. For general structure groups this property can fail. For example, $-id \in SL(2, \mathbb{C})$ is contained in the non-identity component of the centre. This does not admit a central logarithm and hence leads to counter-examples of Proposition \ref{inducedCstaraction} in the case of structure group $SL(2, \mathbb{C})$. 

Let $\phi$ be an $H$-representation, for $H$ a general complex reductive group. Let $S \in H$ be an element of the conjugacy class of the semisimple component of its monodromy. We will say that $\phi$ has `well-behaved' monodromy if $S$ is contained in the identity component of $Z(C_{H}(S))$. The condition of having well-behaved monodromy is necessary and sufficient for Proposition \ref{inducedCstaraction} to hold in this more general setting. Because the centralizer of a semisimple matrix in $GL(m, \mathbb{C})$ is a product of general linear groups and has connected centre, we can iterate the construction of Proposition \ref{inducedCstaraction} in the more general setting of a homogeneous groupoid associated to a $(\mathbb{C}^*)^{k}$-action. 

\begin{corollary} \label{inducedtoric}
Proposition \ref{inducedCstaraction} holds for a general homogeneous groupoid. In particular, a $GL(m, \mathbb{C})$-representation $(P, \phi)$ of a homogeneous groupoid $\mathbb{C}^{k} \ltimes X \to \mathcal{G} \to \mathbb{C}^k$ admits an equivariant $(\mathbb{C}^*)^{k}$-action. 
\end{corollary}

\section{Extending representations}
Let $X$ be a complex manifold, let $D \subset X$ be a Saito free divisor \cite{saito1980theory} and let $U = X \setminus D$. Recall from \cite[Section 3.2]{bischoff2022normal} that this determines a twisted fundamental groupoid $\Pi(X,D)$. The key fact for us is that the category of representations of $\Pi(X,D)$ is \emph{equivalent} to the category of flat connections on $X$ with logarithmic singularities along $D$. Let $D^s \subset D$ denote the singular locus of $D$, and let $X^\times = X \setminus D^s$ and $D^{\times} = D \setminus D^s$. Then $(X^\times,D^{\times})$ defines a smooth Saito free divisor with the property that $X^\times \setminus D^{\times} = X \setminus D = U$. 

In this section we study the problem of extending a logarithmic flat connection on $(X^\times, D^{\times})$ to a logarithmic flat connection on $(X,D)$. We will study this in terms of the corresponding groupoid representations. Recall that the orbits of the groupoid $\Pi(X, D)$ define a stratification of $X$. In \cite{saito1980theory}, this is called the \emph{logarithmic stratification}. The connected components of $U$ and $D^{\times}$ are orbits and hence the singularity locus $D^s$ decomposes as a union of orbits. Therefore, $X^\times$ is a saturated open subset of $X$ and $\Pi(X,D)|_{X^\times} = \Pi(X^\times,D^{\times})$ is a subgroupoid of $\Pi(X,D)$. Restricting representations defines a functor
\[
R : Rep(\Pi(X,D), H) \to Rep(\Pi(X^\times,D^{\times}), H).
\]
The extension problem which we study in this section can be formulated as the problem of characterising the essential image of the functor $R$. When $X = \mathbb{C}^n$ the following lemma shows that this problem reduces to that of determining when the principal bundle $P$ underlying a representation of $\Pi(X^\times,D^{\times})$ is trivializable. 

\begin{lemma} \label{extensionlemma}
The functor $R$ is fully faithful. If $X = \mathbb{C}^n$, then its image consists of representations defined on trivializable bundles. 
\end{lemma}
\begin{proof}
The theorem is trivial if $D^{s}$ is empty. Assuming that it is non-empty, it is an analytic subset of codimension $c \geq 2$. Because $D^{s}$ is a union of orbits, we have that $\Pi(X, D)|_{D^{s}} = s^{-1}(D^{s})$, where $s: \Pi(X,D) \to X$ is the source map. Hence $\Pi(X, D)|_{D^{s}}$ is an analytic subset of $\Pi(X,D)$ of codimension $c$. By Hartogs' theorem \cite[Theorem 5B]{MR0387634} any holomorphic function on $X^\times$ (resp. $\Pi(X^\times,D^{\times})$) extends uniquely to a holomorphic function on $X$ (resp. $\Pi(X,D)$). 

The functor $R$ is clearly faithful since $X^\times$ is dense in $X$. To see that it is full, note that a homomorphism between two representations is locally given by a holomorphic map $T: X^\times \to H$, and so by Hartogs' theorem it extends to $X$. 

If $X = \mathbb{C}^n$, then principal bundles over $X$ (and hence in the image of $R$) are trivializable. Conversely, let $(P, \phi) \in Rep(\Pi(X^\times,D^{\times}), H)$ by a representation such that $P = X^\times \times H$. Then the representation is given by a homomorphism $\phi: \Pi(X^\times,D^{\times}) \to H$. By Hartogs' theorem, it extends to a homomorphism $\phi: \Pi(X,D) \to H$. 
\end{proof}

\subsection*{Weighted homogeneous plane curves}
In this section, we assume that $X = \bb{C}^2$ and is equipped with a $\bb{C}^*$-action which is generated by 
\[
E = px \partial_{x} + q y \partial_{y},
\]
for positive integers $p$ and $q$. We further assume that $D = f^{-1}(0)$, for $f$ a weighted homogeneous function. This means that $E(f) = n f$, for a constant $n$. In this case, the twisted fundamental groupoid $\Pi(X,D)$ is a homogeneous groupoid by \cite[Theorem 3.6]{bischoff2022normal}. Assuming that $D$ is singular, the singularity locus consists of a single point $D^{s} = \{ (0,0) \}$. Hence $X^\times = \bb{C}^2 \setminus \{(0,0)\}. $ The following Theorem, a special case of Mebkhout's \cite[Theorem 10.3-3]{MR2077649}, states that the extension problem may always be solved for $(\mathbb{C}^2,D)$. 

\begin{theorem} \label{extensiontheorem}
Let $H$ be a connected complex reductive group. Let $D \subset \mathbb{C}^2$ be a weighted homogeneous plane curve and let $X^\times = \mathbb{C}^2 \setminus \{ (0,0) \}$. Then restricting representations defines an equivalence of categories 
\[
Rep(\Pi(\mathbb{C}^2,D), H) \cong Rep(\Pi(X^\times,D^{\times}), H).
\]
\end{theorem}
\begin{proof}
By Lemma \ref{extensionlemma}, it suffices to show that for any $(P, \phi) \in Rep(\Pi(X^\times,D^{\times}), H)$, the underlying bundle $P$ is trivializable. Furthermore, it suffices to prove this in the case $H = GL(n, \mathbb{C})$. Indeed, for the general case, let $H \subseteq GL(n, \mathbb{C})$ be an embedding and let $Q = (P \times GL(n, \mathbb{C}))/H$ be the corresponding extension of structure group, so that $(Q, \phi) \in Rep(\Pi(X^\times,D^{\times}), GL(n, \mathbb{C}))$. Then given a trivialization $Q \cong X^\times \times GL(n, \mathbb{C})$, the subbundle $P$ gives rise to a holomorphic map $f_{P} : X^\times \to GL(n, \mathbb{C})/H$. Since $H$ is reductive, the quotient $GL(n, \mathbb{C})/H$ is an affine variety, and hence we may apply Hartogs' theorem to obtain an extension of $f_{P}$ to all $\mathbb{C}^2$. This implies that $P$ extends to a bundle over $\mathbb{C}^2$, which further implies that it is trivializable. 

Assume for the remainder of the proof that $H = GL(n, \mathbb{C})$. We will work with the associated vector bundle $V_{P} = (P \times \mathbb{C}^n)/H$, which is equivalent to $P$. By Proposition \ref{inducedCstaraction}, $V_{P}$ admits an equivariant action of $\mathbb{C}^*$, which lifts the given action on $X^\times$. The `quotient' $V_{P}/\mathbb{C}^*$ is a vector bundle over the `football' $\mathbb{P}_{p,q} = [X^\times/\mathbb{C}^*]$. This is an orbifold: it has coarse moduli space given by $\mathbb{P}^1$, and it has isotropy groups $\mathbb{Z}/p\mathbb{Z}$ over $[1,0]$ and $\mathbb{Z}/q\mathbb{Z}$ over $[0,1]$. We may therefore apply the Grothendieck decomposition theorem (i.e. \cite[Theorem 2.4]{martens2012variations}) to give a decomposition of $V_{P}$ as a direct sum of equivariant line bundles. More precisely, by \cite[Proposition 2.2]{martens2012variations}, the Picard group of $\mathbb{P}_{p,q}$ is generated by an equivariant line bundle $\mathcal{O}(1) = X^\times \times \mathbb{C}$, with $\mathbb{C}^*$-action given by 
\[
\mu \ast (x,y, \lambda) = (\mu^p x, \mu^q y, \mu \lambda). 
\]
Hence, we have a decomposition of $\mathbb{C}^*$-equivariant vector bundles over $X^\times$ given as follows
\[
V_{P} \cong \bigoplus_{i} \mathcal{O}(n_{i}).
\]
In particular, this implies that the underlying holomorphic vector bundle of $V_{P}$ is trivializable. 
\end{proof}
\begin{remark}
For general plane curves, a very similar result was proved by M. Kita in \cite[Proposition 2]{MR0541894}. However, his result is only stated for logarithmic connections obtained from the Deligne-Manin construction. 
\end{remark}
\begin{remark}
Although Theorem \ref{extensiontheorem} is stated for flat connections on $X^\times$, it is in fact a purely local statement which holds for logarithmic flat connections defined on a small punctured neighbourhood of the origin. This is because $\Pi(X^\times, D^\times)$ is Morita equivalent to $\Pi(W^\times, D^\times \cap W)$, for any polydisc $W$ centred at $(0,0)$. More generally, a similar extension result holds for hypersurfaces of the form $\mathbb{C}^k \times D \subset \mathbb{C}^{k + 2}$. 

On the other hand, the assumption that $H$ is reductive cannot be relaxed. For example, let $B \subset GL(2, \mathbb{C})$ be the Borel subgroup of upper triangular matrices. Since $Y = GL(2, \mathbb{C})/B$ is the total space of a $\mathbb{C}^*$-bundle over the projective line, it is not affine and hence we cannot apply Hartogs' theorem to $Y$-valued functions. For a counter-example to the theorem in this setting, consider a non-trivial short exact sequence of toric vector bundles on $\mathbb{CP}^1$: 
\[
0 \to \mathcal{O} \to \mathcal{O}(1)^{\oplus 2} \to \mathcal{O}(2) \to 0. 
\]
Let $P$ be the principal $B$-bundle of frames which split the flag $\mathcal{O} \to \mathcal{O}(1)^{\oplus 2}$. This defines a flat bundle on $X^\times$ with logarithmic poles along the coordinate hyperplanes $D^\times$. This bundle cannot be extended to $\mathbb{C}^2$. Otherwise, we could linearize it to obtain a reduction of structure group to $(\mathbb{C}^*)^2$ and this would induce an isomorphism between $\mathcal{O}(1)^{\oplus 2} $ and $\mathcal{O} \oplus \mathcal{O}(2)$. 
\end{remark}

\section{Castling equivalence} \label{CastlingEq}
In this section we explore another phenomenon which gives rise to bundles which do not extend. This is the phenomenon of castling equivalence for prehomogeneous vector spaces. 

Castling transformation is an operation on linear representations which was introduced by Sato and Kimura in their work \cite{MR430336} classifying irreducible prehomogeneous vector spaces. Recall that a prehomogeneous vector space consists of a linear representation $V$ of an algebraic group $G$ such that there is an open dense orbit. 

Let $G$ be a linear algebraic group, $V$ an $n$-dimensional linear representation and $1 \leq r < n$ an integer. Then $\Hom{\mathbb{C}^r, V}$ is a linear representation of $G \times SL(r, \mathbb{C})$ and $\Hom{\mathbb{C}^{n-r}, V^*}$ is a linear representation of $G \times SL(n-r, \mathbb{C})$. These two representations are said to be related by a castling transformation. More generally, two representations are castling equivalent if they are related by a sequence of castling transformations. In \cite{MR430336} it is shown that castling equivalent representations share a number of features, such as their algebras of polynomial relative invariants, their generic isotropy groups, and the property of being prehomogeneous. Essentially, this is due to the fact that the Grassmanians $Gr(r, V)$ and $Gr(n-r, V^*)$ are isomorphic. 

Let $(V, D)$ be a linear free divisor (cf. \cite[Section 3.2.2]{bischoff2022normal}) and let $G \subset GL(V)$ be the connected component of the group of linear transformations that preserve $D$. Then $(G, V)$ defines a prehomogeneous vector space. Conversely, given a prehomogeneous vector space $(G, V)$, it makes sense to ask whether the complement of the open dense orbit defines a linear free divisor. In \cite{MR2795728} it is shown that the property of defining a linear free divisor is preserved by castling transformations. Hence, we may talk about castling equivalence of linear free divisors. 

The following result shows that the twisted fundamental groupoids of castling equivalent free divisors are `birationally' Morita equivalent. 

\begin{theorem} \label{castleMorita}
Let $(X_{1}, D_{1})$ and $(X_{2}, D_{2})$ be castling equivalent linear free divisors, let $D_{i}^{s}$, $i = 1, 2$, denote the singular loci, and let $X^\times_{i} = X_{i} \setminus D_{i}^{s}$ and $D_{i}^{\times} = D_{i} \setminus D_{i}^{s}$. Then the twisted fundamental groupoids $\Pi(X^\times_{1}, D_{1}^{\times})$ and $\Pi(X^\times_{2}, D_{2}^{\times})$ are Morita equivalent. Furthermore, this Morita equivalence is unique up to rescaling of $X_{1}$. 
\end{theorem}
\begin{proof}
Let $(G \times SL(r, \mathbb{C}), X_{1} = \Hom{\mathbb{C}^r, V})$ and $(G \times SL(n-r, \mathbb{C}), X_{2} = \Hom{\mathbb{C}^{n-r}, V^*})$ be two representations related by a castling transformation. 
Let $I_{r, n} \subset \Hom{\mathbb{C}^r, V}$ be the subspace of injective linear transformations. This is a $G \times SL(r, \mathbb{C})$-invariant subset, which is the complement of a determinental variety $B_{r,n}$ of codimension $n - r + 1 \geq 2$. It follows that both $I_{r,n}$ and $B_{r,n}$ are unions of orbits. If $(G \times SL(r, \mathbb{C}), \Hom{\mathbb{C}^r, V})$ defines a linear free divisor, then for dimension reasons $B_{r,n} \subseteq D^{s}$ and hence $X^\times_{1} \subseteq I_{r,n}$. Furthermore, $\Pi(X_{1}, D_{1}) = (\tilde{G} \times SL(r, \mathbb{C})) \ltimes \Hom{\mathbb{C}^r, V}$, where $\tilde{G}$ is the universal cover of $G$. Analogous statements hold for $(G \times SL(n-r, \mathbb{C}), X_{2} = \Hom{\mathbb{C}^{n-r}, V^*})$. 

We will construct a Morita equivalence between the subgroupoids $(\tilde{G} \times SL(r, \mathbb{C})) \ltimes I_{r,n}$ and $(\tilde{G} \times SL(n-r, \mathbb{C})) \ltimes I_{n-r,n}$, and then show that this induces a Morita equivalence between $\Pi(X^\times_{1}, D_{1}^{\times})$ and $\Pi(X^\times_{2}, D_{2}^{\times})$. 

First, the group $GL(r, \mathbb{C})$ acts freely on $I_{r,n}$ with quotient $Gr(r, V)$, the Grassmanian of $r$-planes in $V$. It follows that the subgroup $SL(r, \mathbb{C})$ also acts freely on $I_{r,n}$. The quotient $L_{r, n} = I_{r,n}/SL(r, \mathbb{C})$ is a principal $\mathbb{C}^*$-bundle over $Gr(r,V)$, identified with the determinant of the tautological rank $r$ vector bundle $T_{r}$. This induces a Morita equivalence between $(\tilde{G} \times SL(r, \mathbb{C})) \ltimes I_{r,n}$ and $\tilde{G} \ltimes L_{r,n}$. Similarly, we have a Morita equivalence between $(\tilde{G} \times SL(n-r, \mathbb{C})) \ltimes I_{n-r,n}$ and $\tilde{G} \ltimes L_{n-r,n}$. We will obtain the desired Morita equivalence by constructing an isomorphism between $L_{r, n}$ and $L_{n-r, n}$. 

The bundles $L_{r,n}$ and $L_{n-r, n}$ have respective bases $Gr(r,V)$ and $Gr(n-r,V^*)$. There is a canonical $\tilde{G}$-equivariant isomorphism between these spaces, given by sending an $r$-plane to its annihilator: 
\[
A: Gr(r,V) \to Gr(n-r ,V^*), \qquad W \mapsto Ann(W). 
\]
Let $T_{r}$ be the canonical $r$-plane bundle over $Gr(r,V)$, and let $T_{n-r}$ be the canonical dual $n-r$-plane bundle over $Gr(n-r,V^*)$. These two bundles are related by the following short exact sequence over $Gr(r,V)$: 
\[
0 \to T_{r} \to V \times Gr(r,V) \to A^{*}T^*_{n-r} \to 0. 
\]
This is a sequence of $\tilde{G}$-equivariant bundles once we equip the trivial bundle $V \times Gr(r,V)$ with the action 
\[
g \ast (v, W) = (g(v), g(W)).  
\]
Taking the determinant of the sequence yields the $\tilde{G}$-equivariant isomorphism 
\[
\phi: L_{r,n} \cong \det(V) \otimes A^{*}(L_{n-r, n}). 
\]
If we choose a volume form on $V$, then $\det(V) \cong \mathbb{C}$, and hence $L_{r,n} \cong L_{n-r, n}$. However, the two actions of $\tilde{G}$ on $L_{r,n}$ differ by the determinant character and this must be corrected. To this end, consider the homomorphism $\tilde{G} \times SL(r, \mathbb{C}) \to GL(X_{1})$ and let $H_{1}$ be its image. Then $\tilde{G} \times SL(r, \mathbb{C}) \to H_{1}$ is the universal cover. Because $H_{1}$ is the connected subgroup of linear transformations that preserve $D_{1}$, it contains a central factor of $\mathbb{C}^{*}$ acting as scaling transformations. This lifts to a homomorphism into the centre $\mathbb{C} \to Z(\tilde{G} \times SL(r, \mathbb{C}))$. Because the centre of $SL(r, \mathbb{C})$ is finite, $\mathbb{C}$ maps purely into $\tilde{G}$. It then follows that $z \in \mathbb{C}$ acts on $V$ as scalar multiplication by $e^{z}$. Consider the character $\det: \tilde{G} \to \mathbb{C}^*$ which is obtained from the action of $G$ on $\det(V)$. This is surjective because $z \in \mathbb{C}$ is sent to $e^{nz}$. Let $K \subset \tilde{G}$ be the connected component of the kernel. Then $\tilde{G} \cong \mathbb{C} \times K$.  Because $K$ is the kernel of the determinant character, the above isomorphism $\phi$ between $L_{r,n}$ and $L_{n-r,n}$ is $K$-equivariant. Furthermore, $z \in \mathbb{C}$ acts on $L_{r,n}$ by multiplication by $e^{rz}$ and on $L_{n-r,n}$ by multiplication by $e^{(r-n)z}$. We therefore have an isomorphism given by 
\[
(\mathbb{C} \times K) \ltimes L_{r,n} \to (\mathbb{C} \times K) \ltimes L_{n-r,n}, \qquad (z, k, v) \mapsto (\frac{r}{r-n} z, k, \phi(v)). 
\]
Note that the only choice in constructing this isomorphism is the trivialisation of $\det(V)$. The different choices are obtained by rescaling, which may be induced by rescaling $X_{1}$. 
A Morita equivalence between groupoids induces a bijection between orbits, and isomorphisms of the corresponding isotropy groups. Therefore, the bijection preserves the codimension of the orbits. Hence, it restricts to a Morita equivalence between $\Pi(X^\times_{1}, D_{1}^{\times})$ and $\Pi(X^\times_{2}, D_{2}^{\times})$. 
\end{proof}

\begin{corollary} \label{castlingrepcat}
Let $\mathcal{C}$ denote a castling equivalence class of linear free divisors. For every $(X, D) \in \mathcal{C}$, the category of $H$-representations $\Rep(\Pi(X^\times, D^{\times}), H)$ only depends on the class $\mathcal{C}$. Hence, we denote it $\Rep(\mathcal{C}, H)$. Furthermore, there is a fully faithful functor
\[
Rep(\Pi(X,D), H) \to \Rep(\mathcal{C}, H). 
\]
In other words, the representations of castling equivalent linear free divisors all embed into a common category. 
\end{corollary}

Let $(V, D)$ be a linear free divisor and assume that the dimension of $V$ is $n \geq 3$. Let $G$ be the connected group of linear transformations which preserve $D$. Then $G = G \times SL(1, \mathbb{C})$ and $V = \Hom{\mathbb{C}, V}$. Hence there is a castling transformation between $(G, V)$ and $(G \times SL(n-1, \mathbb{C}), \Hom{\mathbb{C}^{n-1}, V^{*}})$. In other words, the $n$-dimensional linear free divisor $(V,D)$ is castling equivalent to a linear free divisor $(V', D')$ of dimension $n(n-1)$. As observed in \cite{MR2521436}, this provides a method for constructing infinitely many new castling equivalent linear free divisors of increasing dimensions. The next proposition shows that this also gives rise to a method for constructing non-extendable flat logarithmic connections. 

In order to state the proposition, we need to recall the \emph{residue} of a logarithmic flat connection. Recall that 
\[
\Pi(V', D') \cong (\tilde{G} \times SL(n-1,\mathbb{C})) \ltimes V',
\]
where $\tilde{G}$ is the universal cover of $G$. The origin $0 \in V'$ is an orbit of the groupoid whose isotropy group is $\tilde{G} \times SL(n-1, \mathbb{C})$. Hence, given an $H$-representation $(\phi, P)$ of $\Pi(V', D')$, we can restrict it to the origin to get an $H$-representation of $\tilde{G} \times SL(n-1,\mathbb{C})$. This is the residue, which  can be further restricted to a residual $SL(n-1,\mathbb{C})$ action.

\begin{proposition} \label{castlingembedding}
Given a linear free divisor $(V, D)$ of dimension $dim(V) = n \geq 3$, let $(V',D')$ denote the linear free divisor associated to the castling transform $(G \times SL(n-1, \mathbb{C}), \Hom{\mathbb{C}^{n-1}, V^{*}})$, where $G$ is the connected group of linear transformations of $V$ that preserve $D$. Then there is a fully faithful functor 
\[
F: \Rep(\Pi(V,D), H) \to \Rep(\Pi(V',D'), H)
\]
whose essential image consists of the representations with trivial residual $SL(n-1, \mathbb{C})$ action. 
\end{proposition}
\begin{proof}
The proof of Theorem \ref{castleMorita} constructs a Morita equivalence between $\Pi(I_{1,n}, D)$ and $\Pi(I_{n-1, n}, D')$ and this induces an equivalence $\tilde{F}$ between their categories of representations. In our case, $V \setminus \{0\} = I_{1, n} = L_{1,n}$ and $I_{n-1, n}$ is a principal $SL(n-1, \mathbb{C})$-bundle over $L_{n-1, n} \cong L_{1, n}$. Let $\pi : I_{n-1, n} \to V \setminus \{0\}$ be the bundle projection map. Then the functor $\tilde{F}$ is simply given by pulling a representation back along $\pi^*$. Since the pullback sends the trivial bundle on $I_{1,n}$ to the trivial bundle on $I_{n-1, n}$, by Lemma \ref{extensionlemma} $\tilde{F}$ restricts to the desired fully faithful functor $F$. 

It remains to determine the essential image of $F$. First, note that the for the trivial bundle 
\[
\pi^{*}(I_{1,n} \times H) = I_{n-1, n} \times H
\]
the group $SL(n-1, \mathbb{C})$ acts only on the first factor. This remains the case for the extension of this bundle to $V'$. Hence, the residual action of $SL(n-1, \mathbb{C})$ is trivial. In fact, this is true for all representations in the image of $F$, since they are obtained by pulling back \emph{trivializable} bundles. 

Conversely, let $(P, \phi) \in  \Rep(\Pi(V',D'), H)$ be a representation with trivial residual $SL(n-1, \mathbb{C})$ action. By linearizing the action of $SL(n-1, \mathbb{C}) \ltimes V'$ (eg. via \cite[Theorem 4.8]{bischoff2022normal}), we may choose a trivialization of $P$ such that $\phi: (\tilde{G} \times SL(n-1,\mathbb{C})) \ltimes V' \to H$ satisfies 
\[
\phi(1,s,v) = 1,
\]
for $s \in SL(n-1,\mathbb{C})$ and $v \in V'$. In other words, the action of $SL(n-1, \mathbb{C})$ on $P|_{I_{n-1,n}} \cong I_{n-1, n} \times H$ is trivial on the second factor. Hence, $\tilde{F}^{-1}(P, \phi)$ has underlying bundle given by 
\[
P|_{I_{n-1,n}}/SL(n-1,\mathbb{C}) \cong (I_{n-1, n} \times H)/SL(n-1,\mathbb{C}) = I_{n-1, n}/SL(n-1,\mathbb{C}) \times H = (V \setminus \{0\}) \times H.
\]
This is trivializable and hence $\tilde{F}^{-1}(P, \phi) \in \Rep(\Pi(V,D), H)$.
\end{proof}

\begin{remark}
Note that for a fixed structure group $H$, the sequence of embeddings obtained by repeated application of Proposition \ref{castlingembedding} will eventually stabilise. This is because the dimension of the extra factor of $SL(n, \mathbb{C})$ which gets added at each stage grows quickly, and hence eventually there are no non-trivial morphisms to $H$. In other words, after a certain stage, the functor $F: \Rep(\Pi(V,D), H) \to \Rep(\Pi(V',D'), H)$ becomes an equivalence.
\end{remark}

\begin{example}
Consider the linear free divisor $(\mathbb{C}^3, D)$, where $D$ is the union of coordinate hyperplanes, which is cut out by $xyz = 0$. The connected linear group preserving $D$ is $(\mathbb{C}^*)^3$. Hence, $(\mathbb{C}^3, D)$ is castling equivalent to the prehomogeneous vector space $((\mathbb{C}^*)^3 \times SL(2, \mathbb{C}), \Hom{\mathbb{C}^2, \mathbb{C}^3})$. This defines a linear free divisor $(\mathbb{C}^6, D')$, where $D'$ is the hypersurface cut out by the vanishing of 
\[
f(u_{1},u_{2},v_{1},v_{2},w_{1},w_{2}) = (u_{1}v_{2} - u_{2}v_{1})(v_{1}w_{2} - v_{2}w_{1})(w_{1}u_{2} - w_{2}u_{1}).
\]
By Proposition \ref{castlingembedding} we have an embedding 
\[
\Rep(\Pi(\mathbb{C}^3, D), H) \to \Rep(\Pi(\mathbb{C}^6,D'), H). 
\]
\end{example}

One upshot of Proposition \ref{castlingembedding} is a method for constructing non-extendable logarithmic flat connections. Let $(V, D)$ be a linear free divisor with $\dim(V) = n \geq 3$ and let $(V', D')$ be the castling equivalent free divisor constructed as in Proposition \ref{castlingembedding}. A non-trivial homomorphism $\psi : SL(n-1, \mathbb{C}) \to H$ determines a representation of $\Pi(V', D')$ as follows 
\[
\phi : (\tilde{G} \times SL(n-1,\mathbb{C})) \ltimes V' \to H, \qquad (g,s,v) \mapsto \psi(s). 
\]
Since this has non-trivial residual $SL(n-1, \mathbb{C})$ action, it is not in the image of $F$. On the other hand, by Theorem \ref{castleMorita} it does determine a logarithmic flat connection $(E, \nabla)$ on $(V^\times, D^\times)$, which necessarily does not extend to $(V,D)$.

 \bibliographystyle{plain}

 \bibliography{bibliography.bib}

\begin{thebibliography}{10}

\bibitem{bischoff2022normal}
Francis Bischoff.
\newblock Normal forms and moduli stacks for logarithmic flat connections.
\newblock {\em arXiv preprint arXiv:2209.00631}, 2022.

\bibitem{MR1032917}
A.~A. Bolibrukh.
\newblock The {R}iemann-{H}ilbert problem on the complex projective line.
\newblock {\em Mat. Zametki}, 46(3):118--120, 1989.

\bibitem{MR882000}
A.~Borel, P.-P. Grivel, B.~Kaup, A.~Haefliger, B.~Malgrange, and F.~Ehlers.
\newblock {\em Algebraic {$D$}-modules}, volume~2 of {\em Perspectives in
  Mathematics}.
\newblock Academic Press, Inc., Boston, MA, 1987.

\bibitem{MR2228227}
Ragnar-Olaf Buchweitz and David Mond.
\newblock Linear free divisors and quiver representations.
\newblock In {\em Singularities and computer algebra}, volume 324 of {\em
  London Math. Soc. Lecture Note Ser.}, pages 41--77. Cambridge Univ. Press,
  Cambridge, 2006.

\bibitem{deligne2006equations}
P.~Deligne.
\newblock {\em {\'E}quations diff{\'e}rentielles {\`a} points singuliers
  r{\'e}guliers}, volume 163.
\newblock Springer, 2006.

\bibitem{MR0982153}
N.~P. Erugin.
\newblock The {R}iemann problem.
\newblock {\em Differentsial'nye Uravneniya}, 24(12):2187--2190, 1988.

\bibitem{MR2521436}
Michel Granger, David Mond, Alicia Nieto-Reyes, and Mathias Schulze.
\newblock Linear free divisors and the global logarithmic comparison theorem.
\newblock {\em Ann. Inst. Fourier (Grenoble)}, 59(2):811--850, 2009.

\bibitem{MR2795728}
Michel Granger, David Mond, and Mathias Schulze.
\newblock Free divisors in prehomogeneous vector spaces.
\newblock {\em Proc. Lond. Math. Soc. (3)}, 102(5):923--950, 2011.

\bibitem{MR376680}
Tamafumi Kaneyama.
\newblock On equivariant vector bundles on an almost homogeneous variety.
\newblock {\em Nagoya Math. J.}, 57:65--86, 1975.

\bibitem{MR961215}
Tamafumi Kaneyama.
\newblock Torus-equivariant vector bundles on projective spaces.
\newblock {\em Nagoya Math. J.}, 111:25--40, 1988.

\bibitem{MR0541894}
Michitake Kita.
\newblock The {R}iemann-{H}ilbert problem and its application to analytic
  functions of several complex variables.
\newblock {\em Tokyo J. Math.}, 2(1):1--27, 1979.

\bibitem{MR1024452}
A.~A. Klyachko.
\newblock Equivariant bundles over toric varieties.
\newblock {\em Izv. Akad. Nauk SSSR Ser. Mat.}, 53(5):1001--1039, 1135, 1989.

\bibitem{MR1011361}
A.~A. Klyachko.
\newblock Toric bundles and problems in linear algebra.
\newblock {\em Funktsional. Anal. i Prilozhen.}, 23(2):63--64, 1989.

\bibitem{MR1171283}
A.~A. Klyachko.
\newblock Equivariant vector bundles on toric varieties and some problems of
  linear algebra.
\newblock In {\em Topics in algebra, {P}art 2 ({W}arsaw, 1988)}, volume~26 of
  {\em Banach Center Publ.}, pages 345--355. PWN, Warsaw, 1990.

\bibitem{MR1663340}
Allen Knutson and Eric Sharpe.
\newblock Sheaves on toric varieties for physics.
\newblock {\em Adv. Theor. Math. Phys.}, 2(4):873--961, 1998.

\bibitem{MR1729338}
Allen Knutson and Eric Sharpe.
\newblock Equivariant sheaves.
\newblock volume~10, pages 399--412. 1999.
\newblock Superstrings, M, F, S${\l}ots$ theory.

\bibitem{MR0090529}
I.~A. Lappo-Danilevski\u{\i}.
\newblock {\em Application of matrix functions to the theory of linear systems
  of ordinary differential equations.}
\newblock Gosudarstv. Izdat. Tehn.-Teor. Lit., Moscow,,, 1957.

\bibitem{martens2012variations}
Johan Martens and Michael Thaddeus.
\newblock Variations on a theme of grothendieck.
\newblock {\em arXiv preprint arXiv:1210.8161}, 2012.

\bibitem{MR2077649}
Zoghman Mebkhout.
\newblock Le th\'{e}or\`eme de positivit\'{e}, le th\'{e}or\`eme de comparaison
  et le th\'{e}or\`eme d'existence de {R}iemann.
\newblock In {\em \'{E}l\'{e}ments de la th\'{e}orie des syst\`emes
  diff\'{e}rentiels g\'{e}om\'{e}triques}, volume~8 of {\em S\'{e}min. Congr.},
  pages 165--310. Soc. Math. France, Paris, 2004.

\bibitem{MR0174815}
Josip Plemelj.
\newblock {\em Problems in the sense of {R}iemann and {K}lein.}
\newblock Interscience Publishers John Wiley \& Sons, Inc., New
  York-London-Sydney,, 1964.
\newblock Edited and translated by J. R. M. Radok.

\bibitem{MR0086958}
Helmut R\"{o}hrl.
\newblock Das {R}iemann-{H}ilbertsche {P}roblem der {T}heorie der linearen
  {D}ifferentialgleichungen.
\newblock {\em Math. Ann.}, 133:1--25, 1957.

\bibitem{saito1980theory}
Kyoji Saito.
\newblock Theory of logarithmic differential forms and logarithmic vector
  fields.
\newblock {\em J. Fac. Sci. Univ. Tokyo Sect. IA Math}, 27(2):265--291, 1980.

\bibitem{MR430336}
M.~Sato and T.~Kimura.
\newblock A classification of irreducible prehomogeneous vector spaces and
  their relative invariants.
\newblock {\em Nagoya Math. J.}, 65:1--155, 1977.

\bibitem{MR3237442}
Michele Torielli.
\newblock Deformations of free and linear free divisors.
\newblock {\em Ann. Inst. Fourier (Grenoble)}, 63(6):2097--2136, 2013.

\bibitem{MR0387634}
Hassler Whitney.
\newblock {\em Complex analytic varieties}.
\newblock Addison-Wesley Publishing Co., Reading, Mass.-London-Don Mills, Ont.,
  1972.

\end{thebibliography}

\end{document}